\documentclass[12pt,reqno]{amsart}
\usepackage{enumerate, latexsym, amsmath, amsfonts, amssymb, amsthm, graphicx, color}
\def\pmod #1{\ ({\rm{mod}}\ #1)}

\def\F{\Bbb F}

\def\bg{\bigg}
\def\({\bg(}
\def\){\bg)}

\def\ord{{\rm ord}}

\def\Ack{\medskip\noindent {\bf Acknowledgments}}
\theoremstyle{plain}
\newtheorem{theorem}{Theorem}

\newtheorem{lemma}{Lemma}

\theoremstyle{definition}

\theoremstyle{remark}

\makeatletter
\@namedef{subjclassname@2010}{%
	\textup{2010} Mathematics Subject Classification}
\makeatother
\vspace{4mm}
\begin{document}
	\title[On sumsets involving $k$th powers of finite fields]{On sumsets involving $k$th powers of finite fields}
	
	\author[H.-L. Wu, N.-L. Wei and Y.-B. Li]{Hai-Liang Wu, Ning-Liu Wei and Yu-Bo Li}
	
	\address {(Hai-Liang Wu) School of Science, Nanjing University of Posts and Telecommunications, Nanjing 210023, People's Republic of China}
	\email{\tt whl.math@smail.nju.edu.cn}
	
	\address {(Ning-Liu Wei) School of Science, Nanjing University of Posts and Telecommunications, Nanjing 210023, People's Republic of China}
	\email{weiningliu6@163.com}
	
	\address {(Yu-Bo Li) School of Science, Nanjing University of Posts and Telecommunications, Nanjing 210023, People's Republic of China}
	\email{lybmath2022@163.com}

	\begin{abstract} In this paper, we study some topics concerning the additive decompositions of the set $D_k$ of all $k$th power residues modulo a prime $p$. For example, given a positive integer $k\ge2$, we prove that 
		$$\lim_{x\rightarrow+\infty}\frac{B(x)}{\pi(x)}=0,$$
	where $\pi(x)$ is the number of primes $p\le x$ and $B(x)$ denotes the  cardinality of the set
	$$\{p\le x: p\equiv1\pmod k; D_k\ \text{has a non-trivial 2-additive decomposition}\}.$$
	
	\end{abstract}
	
	\thanks{
		\newline\indent This work was supported by the National Natural Science Foundation of China	(Grant No. 12101321).}

	\maketitle
	\section{Introduction}	
	\setcounter{lemma}{0}
	\setcounter{theorem}{0}
	\setcounter{corollary}{0}
	\setcounter{remark}{0}
	\setcounter{equation}{0}
	\setcounter{conjecture}{0}
	
      Let $p$ be a prime and let $\mathbb{F}_p$ be the finite field of $p$ elements. For any $D\subseteq\mathbb{F}_p$, we say that $D$ has a non-trivial $m$-additive decomposition if 
      $$D=A_1+A_2+\cdots+A_m=\left\{a_1+a_2+\cdots+a_m: a_i\in A_i\ \text{for}\ i=1,2,\cdots,m\right\}$$
      for some $A_1,\cdots, A_m\subseteq\mathbb{F}_p$ with $|A_i|\ge2$ for each $1\le i\le m$.  
      
      S\'{a}rk\"{o}zy \cite{Sarkozy} initiated the study of $2$-additive decompositions involving the set $R_p$ of all quadratic residues modulo $p$.  S\'{a}rk\"{o}zy \cite[Conjecture 1.6]{Sarkozy} conjectured that if $p$ is large enough, then $R_p$ has no non-trivial $2$-additive decompositions. This conjecture is very difficult and only some partial results were obtained. For example, S\'{a}rk\"{o}zy \cite{Sarkozy} showed that if $p$ is sufficiently large and $R_p$ has a non-trivial $2$-additive decomposition $R_p=A+B$, then
      \begin{equation}\label{Eq. the Sarkozy bound}
      	\frac{\sqrt{p}}{3\log{p}}<|A|,|B|<\sqrt{p}\log{p},
      \end{equation} 
      where $|A|$ denotes the cardinality of $A$. By this result, in the same paper, S\'{a}rk\"{o}zy proved that $R_p$ has no non-trivial $3$-additive decompositions if $p$ is large enough. Later (\ref{Eq. the Sarkozy bound}) was improved by several authors (readers may refer to \cite{Chen,Chen and Xi,Shkredov,Shparlinski}). For example, Shkredov  \cite[Corollary 2.6]{Shkredov} obtained 
      \begin{equation}\label{Eq. the Shkredov bound}
      	\left(1/6-o(1)\right)\sqrt{p}\le |A|,|B|\le \left(3+o(1)\right)\sqrt{p}
      \end{equation}
       Moreover, using Fourier analysis method, Shkredov showed that 
       $$R_p\neq A+A$$
       for any $A$ with $|A|\ge2$. In 2021, Chen and Yan \cite[Theorem 1.1]{Chen} further improved (\ref{Eq. the Sarkozy bound}) and showed that for each odd prime $p$, if $R_p=A+B$ with $|A|,|B|\ge2$, then 
       \begin{equation}\label{Eq. the Chen bound}
       	\frac{7-\sqrt{17}}{6}\sqrt{p}+1\le |A|,|B|\le \frac{7+\sqrt{17}}{4}\sqrt{p}-6.63.
       \end{equation}
      Moreover, Chen and Yan also confirmed that $R_p$ has no non-trivial $3$-additive decompositions for any odd prime $p$. Now the best bound was obtained by Chen and Xi \cite{Chen and Xi} which says that 
      \begin{equation}\label{Eq. the Chen-Xi bound}
      	\frac{1}{4}\sqrt{p}+\frac{1}{8}\le |A|,|B|\le 2\sqrt{p}-1
      \end{equation}
       if $R_p=A+B$ with $|A|,|B|\ge2$.
       
	In 2020, Hanson and Petridis \cite{HP} made significant progress on this topic. Fix an integer $k\ge2$. Let $p\equiv1\pmod k$ be a prime and let 
	$$D_k:=\left\{x^k:\ x\in\mathbb{F}_p\setminus\{0\}\right\}$$
	be the set of all $k$th power residues modulo $p$. 
	By using sophisticated polynomial method they showed that if 
	$D_k=A+B$, then $|A||B|=|D_k|$. This is equivalent to saying that 
	$$\left|\{(a,b)\in A\times B: a+b=x\}\right|=1$$
	for any $x\in D_k$. As a direct consequence of this result, we see that 
	$$D_k\neq A+A$$
	for any $A\subseteq\mathbb{F}_p$ with $|A|\ge2$. This generalizes Shkredov's result to $D_k$. 

	Motivated by the above results, in this paper we investigate some topics concerning the additive decompositions of $D_k$. Shparlinski \cite{Shparlinski} proved that if $A+B=D_k$ is a non-trivial $2$-additive decomposition, then 
	$$ \sqrt{p}/k\ll|A|,|B|\ll \sqrt{p}$$
	provided $p$ is large enough. Now we state our first result in which we give the explicit values of the implied constants. This improves Shparlinski's result.
	
	\begin{theorem}\label{Thm. A}
		Let $k\ge2$ be an integer and let $p\equiv1\pmod k$ be a prime greater than $42^2k^2(3k+8)^2$. Suppose that $A+B=D_k$ with $|A|,|B|\ge2$. Then 
		$$\frac{\sqrt{p}}{6k}\le |A|,|B|\le 6\sqrt{p}.$$
	    Moreover, for any $\varepsilon>0$, if $p$ is large enough, then we have 
	    $$\left(\frac{1}{3}-\varepsilon\right)\frac{\sqrt{p}}{k}\le |A|,|B|\le (3+\varepsilon)\sqrt{p}.$$
	\end{theorem}
	By this theorem we can show that if $p$ is sufficiently large in terms of $k$, then $D_k$ has no non-trivial $3$-additive decompositions.
	\begin{theorem}\label{Thm. B}
			Let $k\ge2$ be an integer and let $p\equiv1\pmod k$ be a prime greater than $217^2k^4$. Then 
			$$D_k\neq A+B+C$$
			for any $A,B,C$ with $|A|,|B|,|C|\ge2$. 
	\end{theorem}
	
Hanson and Petridis \cite[Corollary 1.4]{HP} showed that 
$$\left|\{p\le x:\ D_2\ \text{has a non-trivial $2$-additive decomposition}\}\right|=o(\pi(x)),$$
where $\pi(x)$ denotes the number of primes $p \le x$. 
	
By slightly modifying the method introduced by Shakan and by Theorem \ref{Thm. A}, we can generalize the result of Hanson and Petridis to $k$th power residues. For simplicity, given an integer $k\ge2$, a prime $p\equiv1\pmod k$ is said to be {\it bad} if $D_k$ has a non-trivial $2$-additive decomposition. The following result implies that bad primes are sparse in the set of all primes. 
\begin{theorem}\label{Thm. C}
	Let $k\ge2$ be an integer. Then 
	$$\lim_{x\rightarrow+\infty}\frac{\left|\{p\le x: p\equiv1\pmod k;\ p\  \text{is bad}\}\right|}{\pi(x)}=0.$$
	Moreover, we have 
	$$\left|\{p\le x: p\equiv1\pmod k;\ p\  \text{is bad}\}\right|\ll_k 
	\int_{e^2}^x\left(\log t\right)^{-1-\delta}\left(\log\log t\right)^{\frac{-3}{2}}dt,$$
	where $\delta=1-\frac{1+\log\log2}{\log2}=0.086\cdots$.
\end{theorem}
	
The outline of this paper is as follows. In Section 2 we shall introduce some necessary lemmas and the proofs of the main results will be given in Sections 3--4. In Section 5 we will pose some problems for further research.
	
    \section{Notations and Preparations}	
    \setcounter{lemma}{0}
    \setcounter{theorem}{0}
    \setcounter{corollary}{0}
    \setcounter{remark}{0}
    \setcounter{equation}{0}
    \setcounter{conjecture}{0}
    
   We first introduce some notations. Given a prime $p$, let $\mathbb{F}_p^{\times}=\{x\in\mathbb{F}_p:\ x\neq0\}$ and let $\widehat{\mathbb{F}_p^{\times}}$ be the group of all multiplicative characters of $\mathbb{F}_p$. For each $\chi\in\widehat{\mathbb{F}_p^{\times}}$ we define $\chi(0)=0$ and let  $\ord(\chi)$ denote the order of $\chi$. Now fix a character $\psi\in\widehat{\mathbb{F}_p^{\times}}$ with $\ord(\psi)=k$. One can verify that 
   \begin{equation}\label{Eq. characteristic functions for kth power residues}
   	f_k(x):=\sum_{i=1}^{k-1}\psi^i(x)=
   	\begin{cases}
   		k-1   & \mbox{if}\ x\in D_k,\\
   		0     & \mbox{if}\ x=0,\\
   		-1    & \mbox{otherwise}.
   	\end{cases}
   \end{equation}
   For simplicity, $\sum_{x\in\mathbb{F}_p}$ will be abbreviated as $\sum_{x}$. 
    
    We begin with the following result which is known as the Weil bound (see  \cite[Theorem 5.41]{LN}).
    
   \begin{lemma}\label{Lem. Weil's theorem}
   	Let $\chi\in\widehat{\F_p^{\times}}$ with $\ord(\chi)=m>1$ and let $f(x)\in\F_p[x]$ be a monic polynomial which is not of the form $g(x)^m$ for any $g(x)\in\F_p[x]$. Then
    	\begin{equation*}
    		\left|\sum_{x}\chi(f(x))\right|\le (r-1)p^{1/2},
    	\end{equation*}
    where $r$ is the number of distinct roots of $f(x)$ in the algebraic closure $\overline{\F_p}$.
   \end{lemma}
    
  \begin{lemma}\label{Lem. both A and B must contain at least 3 elements}
  	Let $k\ge2$ be an integer and let $p\equiv1\pmod k$ be a prime greater than $16k^2$. Suppose that $D_k=A+B$ for some subsets $A,B$ with $|A|\ge|B|\ge2$. Then $|B|\ge3$.  
  \end{lemma}  
    
  \begin{proof}
   	It is clear that $(A+z)+(B-z)=A+B$ for any $z\in\mathbb{F}_p$. Hence by shifting we may assume that $0\in B$ and hence $A\subseteq D_k$. 
   	
   	Suppose $|B|=2$ and let $B=\{0,t\}$ for some $t\neq0$. For any $x\in D_k$, we claim that either $x+t$ or $x-t$ is contained in $D_k$. In fact, if $x\in A$, then $x+t\in D_k$ follows directly from $A+B=D_k$. In the case $x\not\in A$, as $(x-B)\cap A\neq\emptyset$, we have $x-t\in A\subseteq D_k$. 
   	
   Now define
   	\begin{equation*}
   		G(x):=\left(k-1-f_k(x+t)\right)\left(k-1-f_k(x-t)\right)\left(1+f_k(x)\right).
   	\end{equation*}
    By (\ref{Eq. characteristic functions for kth power residues}) and the above we see that $G(x)$ vanishes for any $x\in\mathbb{F}_p^{\times}$. Hence by (\ref{Eq. characteristic functions for kth power residues}) again we obtain 
    \begin{equation}\label{Eq. A in Lem. both A and B must contain at least 3 elements}
    	0\le \sum_{x}G(x)=G(0)\le k^2.
    \end{equation}   	
   	
On the other hand, noting that $\sum_{x}f_k(x)=0$, we obtain that $\sum_{x}G(x)$ is equal to 
    \begin{align*}
    	(k-1)^2p&-2(k-1)\sum_{x}f_k(x)f_k(x+t)\\
    	&+\sum_{x}f_k(x)f_k(x+2t)+\sum_{x}f_k(x)f_k(x+t)f_k(x+2t).
    \end{align*}
Clearly, 
   	\begin{equation*}
   	\sum_{x}f_k(x)f_k(x+t)=\sum_{i_0,i_1\in[1,\ k-1]}\sum_{x}\psi^{i_0}(x)\psi^{i_1}(x+t).
   	\end{equation*}
As $t\neq0$ and $i_0,i_1\in[1,k-1]$, the polynomial $x^{i_0}(x+t)^{i_1}\neq h(x)^k$ for any $h(x)\in\mathbb{F}_p[x]$. Hence by Lemma \ref{Lem. Weil's theorem} we obtain 
   	\begin{equation}\label{Eq. B in Lem. both A and B must contain at least 3 elements}
   		\left|\sum_{x}f_k(x)f_k(x+t)\right|\le (k-1)^2\sqrt{p}. 
   	\end{equation}
With the same reason we also have 
   	\begin{equation}\label{Eq. C in Lem. both A and B must contain at least 3 elements}
   		\left|\sum_{x}f_k(x)f_k(x+2t)\right|\le (k-1)^2\sqrt{p},
   	\end{equation}
and 
   \begin{equation}\label{Eq. D in Lem. both A and B must contain at least 3 elements}
   	\left|\sum_{x}f_k(x)f_k(x+t)f_k(x+2t)\right|\le 2(k-1)^3\sqrt{p}. 
   \end{equation}
Combining (\ref{Eq. B in Lem. both A and B must contain at least 3 elements})--(\ref{Eq. D in Lem. both A and B must contain at least 3 elements}) with (\ref{Eq. A in Lem. both A and B must contain at least 3 elements}), by computations we obtain that if $p\ge 16k^2$, then 
\begin{equation*}
	k^2\ge \sum_{x}G(x)\ge (k-1)^2p-(4(k-1)^3+(k-1)^2)\sqrt{p}>k^2,
\end{equation*}
which is a contradiction. This completes the proof.
\end{proof} 
    
 The next result illustrates that for a sufficiently large prime $p\equiv1\pmod k$, if $A+B=D_k$ with $|A|,|B|\ge2$, then both $|A|$ and $|B|$ are also sufficiently large.    
    
 \begin{lemma}\label{Lem. A and B are large if p is large}
 	Let $k\ge2$ be an integer and let $p\equiv 1\pmod k$ be a prime. Suppose that $A+B=D_k$ with $|A|\ge |B|\ge2$. Then 
 	$$|B|\ge\sqrt{\frac{(k^2-3)p-k^2}{k^2(k-1)\sqrt{p}}}.$$
 	Hence $|A|$ and $|B|$ are sufficiently large provided that $p$ is large enough. 
 \end{lemma}  
    
 \begin{proof}
 	By Lemma \ref{Lem. both A and B must contain at least 3 elements} we may set $B=\{y_1,\cdots,y_r\}$ with $r\ge3$. Define 
 	$$H(x)=\frac{1}{k^r}\prod_{j=1}^r\left(1+f_k(x+y_j)\right).$$
 	Clearly $H(x)\ge0$ and $H(x)=1$ for any $x\in A$. Hence $|A|\le\sum_{x}H(x)$. Now we consider $\sum_{x}H(x)$. One can verify that $\sum_{x}H(x)$ is equal to
 	\begin{align*}
 		 &\frac{1}{k^r}\sum_{x}\prod_{j=1}^r\left(1+f_k(x+y_j)\right)\\
 		=&\frac{1}{k^r}\sum_{x}\left(1+\sum_{j=1}^rf_k(x+y_j)
 	 	  +\sum_{i=2}^r\sum_{1\le j_1<\cdots<j_i\le r}\prod_{s=1}^if_k(x+y_{j_s})\right).
 	\end{align*}
 	Note that for any $i\ge2$ and $1\le j_1<\cdots<j_i\le r$ we have 
 	$$\sum_{x}\prod_{s=1}^if_k(x+y_{j_s})
 	=\sum_{1\le t_1,\cdots,t_i\le k-1}\sum_{x}\prod_{s=1}^i\psi^{t_s}(x+y_{j_s}).$$
As $1\le t_j\le k-1$ for each $1\le j\le i$, we have 
$$\prod_{s=1}^i(x+y_{j_s})^{t_s}\neq g(x)^k$$
for any $g(x)\in\mathbb{F}_p[x]$ and hence by Lemma \ref{Lem. Weil's theorem}, for any $2\le i\le r$ we see that  
$$\left|\sum_{x}\prod_{s=1}^if_k(x+y_{j_s})\right|\le (k-1)^i(i-1)\sqrt{p}.$$
This, together with $\sum_{x}f_k(x+y_j)=0$ and $|A||B|\ge(p-1)/k$, implies that 
\begin{align*}
	\frac{p-1}{rk}\le |A|\le\sum_{x}H(x)
&\le\frac{1}{k^r}\left(p+\sqrt{p}\sum_{i=2}^r\binom{r}{i}(k-1)^i(i-1)\right)\\
&=\frac{1}{k^r}\left(p+\sqrt{p}\left(rk^r-rk^{r-1}-k^r+1\right)\right),
\end{align*} 	
 where the last equality follows from the identities 
 $$\sum_{i=2}^r\binom{r}{i}i(k-1)^i=(rk^{r-1}-r)(k-1),$$
 and 
 $$\sum_{i=2}^r\binom{r}{i}(k-1)^i=k^r-1-r(k-1).$$
In view of the above, we obtain 
\begin{equation}\label{Eq. A in the Lem. A and B are large if p is large}
	\frac{p-1}{rk}\le \frac{1}{k^r}\left(p+\sqrt{p}\left(rk^r-rk^{r-1}-k^r+1\right)\right)
	\le\frac{p}{k^r}+\frac{r(k-1)\sqrt{p}}{k}.
\end{equation}

As $r\ge3$ by Lemma \ref{Lem. both A and B must contain at least 3 elements}, we have $r/k^r\le 3/k^3$. Combining this with (\ref{Eq. A in the Lem. A and B are large if p is large}) we obtain 
$$r\ge \sqrt{\frac{(k^2-3)p-k^2}{k^2(k-1)\sqrt{p}}}.$$
This completes the proof.
\end{proof}   
    
For any function $g:\mathbb{F}_p\rightarrow\mathbb{C}$ and any positive integer $r$, we define a function $C_{r+1}(g):\mathbb{F}_p^r\rightarrow\mathbb{C}$ by 
\begin{equation}\label{Eq. definition of C(g)}
	C_{r+1}(g)(x_1,\cdots,x_r):=\sum_{x}g(x)g(x+x_1)\cdots g(x+x_r).
\end{equation}
Also, for any functions $f,g: \mathbb{F}_p\rightarrow\mathbb{C}$, the function $f\circ g$ is defined by 
\begin{equation}\label{Eq. definition of convolution}
	(f\circ g)(x):=\sum_{y}f(y)g(x+y).
\end{equation}
 
We conclude this section with the following result obtained by Shkredov \cite[Lemma 2.1]{Shkredov}. 

\begin{lemma}\label{Lem. Shkredov on C(f)}
For any functions $f,g: \mathbb{F}_p\rightarrow\mathbb{C}$ and any positive integer $r$, 
	\begin{equation*}
		\sum_{x}(f\circ g)^{r+1}(x)
		=\sum_{x_1,\cdots,x_r}C_{r+1}(f)(x_1,\cdots,x_r)\cdot C_{r+1}(g)(x_1,\cdots,x_r).
	\end{equation*}
\end{lemma} 
    
\section{Proofs of Theorems \ref{Thm. A}--\ref{Thm. B}}  
\setcounter{lemma}{0}
\setcounter{theorem}{0}
\setcounter{corollary}{0}
\setcounter{remark}{0}
\setcounter{equation}{0}
\setcounter{conjecture}{0}  
    
 For any $A\subseteq\mathbb{F}_p$, the characteristic function of $A$ is denoted by $1_A$, i.e., 
 \begin{equation}\label{Eq. characteristic function of A}
 	1_A(x)=\begin{cases}
 		1 & \mbox{if}\ x\in A,\\
 		0 & \mbox{if}\ x\not\in A. 
 	\end{cases}
 \end{equation} 
Now we make use of the method introduced by Shkredov to prove our first result. 
 
{\noindent\bf Proof of Theorem \ref{Thm. A}.}  As the cases $k=2$ and $k=3$ were proved by Chen and Yan \cite{Chen} and Wu and She \cite{Wu-She} respectively, in this proof, we assume $k\ge4$. Suppose that $A+B=D_k$ with $|A|,|B|\ge2$. Let $|A|=a, |B|=b$. We begin with the following equalities.
\begin{align*}
	\sum_{x\in B}\left(1_A\circ f_k\right)^4(x)
	&=\sum_{x\in B}\sum_{y_1,y_2,y_3,y_4}\prod_{j=1}^{4}1_A(y_j)f_k(x+y_j)\\
	&=\sum_{x\in B}\sum_{y_1,y_2,y_3,y_4\in A}\prod_{j=1}^{4}f_k(x+y_j)\\
	&=(k-1)^4a^4b.
\end{align*}
Note that $(1_A\circ f_k)^4(x)\ge0$ for any $x\in\mathbb{F}_p$. Hence by Lemma \ref{Lem. Shkredov on C(f)} we obtain 
\begin{equation}\label{Eq. A in the proof of Thm. A}
(k-1)^4a^4b\le\sum_{x}\left(1_A\circ f_k\right)^4(x)
=\sum_{x_1,x_2,x_3}C_4(1_A)(x_1,x_2,x_3)C_4(f_k)(x_1,x_2,x_3).
\end{equation}    

We next consider $C_4(f_k)(x_1,x_2,x_3)$ which is equal to 
\begin{align*}
	 &\sum_{x}f_k(x)f_k(x+x_1)f_k(x+x_2)f_k(x+x_3)\\
	=&\sum_{0<i_0,i_1,i_2,i_3<k }\sum_{x}\psi^{i_0}(x)\psi^{i_1}(x+x_1)\psi^{i_2}(x+x_2)\psi^{i_3}(x+x_3).
\end{align*}
Let $$E=\{(x,x,0),(x,0,x),(0,x,x):\ x\neq 0\}\cup\{(0,0,0)\}.$$ 
As $0<i_0,i_1,i_2,i_3<k$, one can verify that given any $(x_1,x_2,x_3)\not\in E$, the polynomial 
$$x^{i_0}(x+x_1)^{i_1}(x+x_2)^{i_2}(x+x_3)^{i_3}\neq g(x)^{k}$$
for any $g(x)\in\mathbb{F}_p[x]$. Hence by Lemma \ref{Lem. Weil's theorem} we obtain 
\begin{equation}\label{Eq. B in the proof of Thm. A}
	\left|C_4(f_k)(x_1,x_2,x_3)\right|\le 3(k-1)^4\sqrt{p}
\end{equation}
whenever $(x_1,x_2,x_3)\not\in E$. 

Now we consider the case $(x_1,x_2,x_3)\in E\setminus\{(0,0,0)\}$. Suppose first that $(x_1,x_2,x_3)=(0,y,y)$ for some $y\neq0$. Then the polynomial 
$x^{i_0+i_1}(x+y)^{i_2+i_3}$ is a $k$th power of some $g(x)\in\mathbb{F}_p[x]$ if and only if both $i_0+i_1$ and $i_2+i_3$ are equal to $k$. Hence by Lemma \ref{Lem. Weil's theorem} again we obtain that
\begin{equation}\label{Eq. C in the proof of Thm. A}
	\left|C_4(f_k)(x_1,x_2,x_3)\right|\le (k-1)^2p+((k-1)^4-(k-1)^2)\sqrt{p}
\end{equation}
whenever $(x_1,x_2,x_3)\in E\setminus\{(0,0,0)\}$. Suppose now  $(x_1,x_2,x_3)=(0,0,0)$. Then the polynomial $x^{i_0+i_1+i_2+i_3}$ is a $k$th power of some $g(x)\in\mathbb{F}_p[x]$ if and only if $i_0+i_1+i_2+i_3\in\{k,2k,3k\}$. Also, if $i_0+i_1+i_2+i_3\not\in\{k,2k,3k\}$, then $\psi^{i_0+i_1+i_2+i_3}$ is not a trivial character and hence 
$$\sum_{x}\psi(x)^{i_0+i_1+i_2+i_3}=0.$$
By the above we obtain 
\begin{equation}\label{Eq. D in the proof of Thm. A}
|C_4(f_k)(0,0,0)|\le 3(k-1)^3p.
\end{equation}

Combining (\ref{Eq. B in the proof of Thm. A})--(\ref{Eq. D in the proof of Thm. A}) with (\ref{Eq. A in the proof of Thm. A}), by computations we obtain that 
\begin{align*}
     &(k-1)^4a^4b\le\\
  &3(k-1)^4\sqrt{p}a^4+3(k-1)^2pa^2+3((k-1)^4-(k-1)^2)\sqrt{p}a^2+3(k-1)^3pa.
\end{align*}
By \cite[Corollary 1.3]{HP} we have $ab=(p-1)/k$ and hence we have 
\begin{equation*}
\frac{p-1}{k}a\le 3\sqrt{p}a^2+\frac{3p}{(k-1)^2}+3\left(1-\frac{1}{(k-1)^2}\right)\sqrt{p}+\frac{3p}{a(k-1)}.
\end{equation*}
As $a\ge3$ (by Lemma \ref{Lem. both A and B must contain at least 3 elements}) and $3/(k-1)^2+1/(k-1)\le 8/k$ for any $k\ge2$, we further obtain the inequality
\begin{equation}\label{Eq. E in the proof of Thm. A}
	3k\sqrt{p}a^2-(p-1)a+8p+3k\sqrt{p}\ge0.
\end{equation}   
When $p>42^2k^2(3k+8)^2$, the inequality (\ref{Eq. E in the proof of Thm. A}) implies either
$$a\ge\frac{p-1+\sqrt{\Delta}}{6k\sqrt{p}}$$
or
$$a\le\frac{p-1-\sqrt{\Delta}}{6k\sqrt{p}}=
\frac{12k\sqrt{p}(8p+3k\sqrt{p})}{6k\sqrt{p}(p-1+\sqrt{\Delta})}<
\frac{12k\sqrt{p}(8p+3k\sqrt{p})}{6k\sqrt{p}(p-1)}\le 17,$$
where $\Delta=(p-1)^2-12k\sqrt{p}(8p+3k\sqrt{p})>0$ for any $p>42^2k^2(3k+8)^2$. 

By Lemma \ref{Lem. A and B are large if p is large} we see that $a>18$ whenever $p>42^2k^2(3k+8)^2$. Hence if $p>42^2k^2(3k+8)^2$, then we have 
\begin{equation}\label{Eq. F in the proof of Thm. A}
a\ge\frac{p-1+\sqrt{\Delta}}{6k\sqrt{p}}\ge \frac{\sqrt{p}}{6k}.
\end{equation} 
Moreover, for any $\varepsilon>0$, if $p$ is sufficiently large, then we have 
  \begin{equation}\label{Eq. G in the proof of Thm. A}
  	a\ge\frac{p-1+\sqrt{\Delta}}{6k\sqrt{p}}\ge
  	\left(\frac{1}{3}-\varepsilon\right)\frac{\sqrt{p}}{k}.
  \end{equation}  
By \cite[Corollary 1.3]{HP} which says that $ab=(p-1)/k$, we obtain that $b\le 6\sqrt{p}$ if $p>42^2k^2(3k+8)^2$. Similarly, if $p$ is large enough, then we further have $b\le (3+\varepsilon)\sqrt{p}$. With the same method, one can easily obtain that $a\le 6\sqrt{p}$ for $p>42^2k^2(3k+8)^2$ and $a\le (3+\varepsilon)\sqrt{p}$ if $p$ is large enough. 

This completes the proof.\qed
   
 In order to prove our next result, we need the following result due to Ruzsa \cite[Theorem 5.1]{R}, which was later generalized by Gyarmati, Matolcsi and Ruzsa \cite[Theorem 1.2]{GMR}.
 \begin{lemma}\label{Lem. GMR}
 	Let $p$ be an odd prime and let $A,B,C\subseteq\mathbb{F}_p$. Then
 	$$\left|A+B+C\right|^{2}\le|A+B||B+C||C+A|.$$
 \end{lemma}
   
{\noindent\bf Proof of Theorem \ref{Thm. B}.}  Since the cases $k=2$ and $k=3$ were proved by Chen and Yan \cite{Chen} and Wu and She \cite{Wu-She} respectively, we assume $k\ge4$ in this proof. Suppose $D_k=A+B+C$ for some subsets $A,B,C$ with $|A|,|B|,|C|\ge2$. Then by Theorem \ref{Thm. A} and Lemma \ref{Lem. GMR} we have 
$$\left(\frac{p-1}{k}\right)^2=|A+B+C|^2\le |A+B||B+C||C+A|\le 216p^{3/2},$$
which is a contradiction whenever $p>217^2k^4$. 

This completes the proof. \qed   
    
\section{Proof of Theorem \ref{Thm. C}}
 We first state some results due to Ford \cite{Ford} which concern the number of integers with a divisor in a fixed interval. For any positive integer $n$ and any real numbers $y,z$ with $y\le z$, let 
 \begin{equation}\label{Eq. tau(n,y,z)}
 	\tau(n,y,z):=\left|\{d:\ d\mid n;\ y<d\le z\}\right|
 \end{equation}
 and let
 \begin{equation}\label{Eq. H(x,y,z)}
 	H(x,y,z):=\left|\{n\le x:\ \tau(n,y,z)\ge1\}\right|.
 \end{equation}  
  Also for any integer $\lambda$, we define 
  \begin{equation}\label{Eq. H(x,y,z,P)}
  	H(x,y,z,P_{\lambda}):=\left|\{n\le x:\ n\in P_{\lambda};\ \tau(n,y,z)\ge1\}\right|,
  \end{equation}
  where $P_{\lambda}=\{p+\lambda:\ p\ \text{is a prime}\}$. 
  
 The next result obtained by Ford \cite[Theorem 1 and Theorem 6]{Ford} will play an important role in our proof.
  
  \begin{lemma}\label{Lem. the Ford result}
  	{\rm (i)} Let $u=u(y)\ge0$ be a real function with $\lim_{y\rightarrow+\infty}u(y)=0$. If $x>10^6$ and $200\le 2y\le z\le y^2$, then 
  	$$\frac{H(x,y,z)}{x}\ll u^{\delta}\left(\log\frac{2}{u}\right)^{-3/2},$$ 
  	where $\delta=1-\frac{1+\log\log2}{\log2}=0.086\cdots$.

  	{\rm (ii)} Let $\lambda$ be a non-zero integer and let $x>0$ be a real number. Let $1\le y\le \sqrt{x}$ and $y+(\log y)^{2/3}\le z\le x$. Then 
  	$$H(x,y,z,P_{\lambda})\ll_{\lambda}\frac{H(x,y,z)}{\log x}.$$
  \end{lemma}
  
  By Theorem \ref{Thm. A} and by slightly modifying the method of Shakan, we can prove our last result.
  
  {\noindent\bf Proof of Theorem \ref{Thm. C}.}  Fix an integer $k\ge2$. Suppose that $x$ is large enough. Let 
  $$A(x)=\left\{\frac{x}{e}\le p\le x:\ p\equiv 1\pmod k;\ p\ \text{is a bad prime}\right\}.$$ 
  Suppose $p\in A(x)$. Then there are $A,B$ with $|A|=a\ge |B|=b\ge2$ such that $D_k=A+B$. By \cite[Corollary 1.3]{HP} and Theorem \ref{Thm. A} we have 
  $$p-1=kab$$ 
  and $$b\in\left[\frac{\sqrt{p}}{6k},\sqrt{p}\right]\subseteq\left(\frac{\sqrt{x}}{100k},\sqrt{x}\right].$$
  This implies that 
  $$A(x)\subseteq \left\{n\le x:\ n\in P_{-1};\ \tau\left(n,\sqrt{x}/(100k),\sqrt{x}\right)\ge1\right\}$$
  and hence 
  $$|A(x)|\le H\left(x,\frac{\sqrt{x}}{100k},\sqrt{x},P_{-1}\right).$$
  Now applying Lemma \ref{Lem. the Ford result}(ii), we obtain 
  \begin{equation}\label{Eq. A in the proof of Thm. C}
  	 |A(x)|\le H\left(x,\frac{\sqrt{x}}{100k},\sqrt{x},P_{-1}\right)\ll\frac{H\left(x,\frac{\sqrt{x}}{100k},\sqrt{x}\right)}{\log x}.
  \end{equation}
Moreover, let $u(y)=\log(100k)/\log y$. 
By Lemma \ref{Lem. the Ford result}(i) we have 
\begin{equation}\label{Eq. B in the proof of Thm. C}
H\left(x,\frac{\sqrt{x}}{100k},\sqrt{x}\right)\ll_k x(\log x)^{-\delta}(\log\log x)^{-3/2}.
\end{equation} 
Combining (\ref{Eq. A in the proof of Thm. C}) with (\ref{Eq. B in the proof of Thm. C}) we obtain 
$$|A(x)|\ll_k \frac{x}{\log x}(\log x)^{-\delta}(\log\log x)^{-3/2}=o\left(\frac{x}{\log x}\right).$$ In view of the above, we have
\begin{align*}
	 &\left|\left\{p\le x: p\equiv1\pmod k;\ p\ \text{is a bad prime}\right\}\right|\\
	=&\sum_{s=2}^{\log x}\left|\left\{e^{s-1}\le p\le e^s: p\equiv1\pmod k;\ p\ \text{is a bad prime}\right\}\right|\\
	=&\sum_{s=2}^{\log x}|A(e^s)|\ll_k\sum_{s=2}^{\log x}\frac{e^s}{s}s^{-\delta}(\log s)^{-3/2}
	\ll_k\int_{2}^{\log x}\frac{e^t}{t}t^{-\delta}(\log t)^{-3/2}dt\\
	=&	\int_{e^2}^x\left(\log t\right)^{-1-\delta}\left(\log\log t\right)^{\frac{-3}{2}}dt=o\left(\int_{e^2}^x\frac{dt}{\log t}\right)\\
	=&o(\pi(x)).
\end{align*}
The last equality follows from the prime number theorem. 

In view of the above, we have completed the proof of Theorem \ref{Thm. C}. \qed

\section{Concluding Remarks}

Let $p$ be a prime. An element $g\in\mathbb{F}_p$ is called a {\it primitive element} of $\mathbb{F}_p$ if $g$ is a generator of the cyclic group $\mathbb{F}_p^{\times}$. Let 
$$\mathcal{P}_p:=\{g\in\mathbb{F}_p:\ g\ \text{is a primitive element}\}.$$
Dartyge and S\'{a}rk\"{o}zy \cite{Dartyge} conjectured that $\mathcal{P}_p$ has no non-trivial $2$-additive decompositions provided that $p$ is large enough. As the S\'{a}rk\"{o}zy conjecture, this conjecture is also out of reach. 
Dartyge and S\'{a}rk\"{o}zy \cite[Theorem 2.1]{Dartyge} proved that if $p$ is large enough and $A+B=\mathcal{P}_p$ with $|A|,|B|\ge2$, then
$$\frac{\varphi(p-1)}{\tau(p-1)\sqrt{p}\log{p}}<|A|,|B|<\tau(p-1)\sqrt{p}\log{p},$$
where $\varphi$ is the Euler function and $\tau(p-1)$ denotes the number of positive divisors of $p-1$. Later this bound was improved by Shparlinski \cite{Shparlinski}. Recently, She and Wu \cite{she-wu} further improved this bound and obtained that 
$$\left(\frac{1}{2}-o(1)\right)\frac{\varphi(p-1)}{\sqrt{p}}\le |A|,|B|\le (2+o(1))\sqrt{p}.$$
Compared to the additive decompositions of $D_k$, it is difficult for us to obtain some qualitative results concerning the $2$-additive decompositions of $\mathcal{P}_p$. For example, it seems that the methods introduced by Shkredov, Hanson and Petridis cannot be applied to prove $\mathcal{P}_p\neq A+A$. 

On the other hand, as a counterpart of Theorem \ref{Thm. C}, it is also difficult to prove 
$$\lim_{x\rightarrow+\infty}\frac{\left|\left\{p\le x: \mathcal{P}_p\ \text{has a non-trivial $2$-additive decompositions}\right\}\right|}{\pi(x)}=0.$$ 
   
	\Ack\ We would like to thank the referees for their helpful comments. We also thank Prof. Zhi-Wei Sun, Prof. Hao Pan and Dr. Yue-Feng She for their steadfast encouragements.

\end{document}